\documentclass{amsart}%
\usepackage{amsmath}
\usepackage{amssymb}
\usepackage{amsfonts}
\usepackage{graphicx}%
\providecommand{\U}[1]{\protect\rule{.1in}{.1in}}
\newtheorem{theorem}{Theorem}[section]
\theoremstyle{plain}

\newtheorem{corollary}{Corollary}[section]

\newtheorem{lemma}{Lemma}[section]

\theoremstyle{definition}
\newtheorem{definition}{Definition}[section]
\newtheorem{remark}{Remark}[section]
\numberwithin{equation}{section}
\newtheorem*{remark*}{Remark}
\numberwithin{equation}{section}
\newtheorem*{section*}{Acknowledgments}

\usepackage{bm}
\newcommand*{\B}[1]{\ifmmode\bm{#1}\else\textbf{#1}\fi}

\usepackage[hyphens]{url}

\usepackage{xcolor} 
\definecolor{orange}{rgb}{1,0.5,0}
\definecolor{Ggreen}{rgb}{0.,0.775,0.0128}
\definecolor{Bblue}{rgb}{0.16,.32,0.91}
\usepackage[hyphens]{url}

 \usepackage[backref=page]{hyperref}
\hypersetup{
    colorlinks = true,
linkcolor={red},
urlcolor={blue},
citecolor={Ggreen},
urlcolor = {blue},
citebordercolor = {0.33 .58 0.33},
 linkbordercolor = {0.99 .28 0.23},
 breaklinks=true
}
\usepackage{cite}

\makeatletter
\@namedef{subjclassname@2020}{\textup{2020} Mathematics Subject Classification}
\makeatother

\begin{document}
\title[Lindemann-Weierstrass theorem]{A simple and self-contained proof for the Lindemann-Weierstrass theorem}
\author{Sever Angel Popescu}
\address{Technical University of Civil Engineering Bucharest, Department of Mathematics
and Computer Science, B-ul Lacul Tei 122, sector 2, Bucharest 020396,
Bucharest, Romania\\
 }
\email{angel.popescu@gmail.com}
\date{July 29, 2023}
\dedicatory{Dedicated to the memory of our Professor Nicolae Popescu.}
\begin{abstract}
The famous result of Lindemann and Weierstrass says that if $a_{1}%
,a_{2},\ldots,a_{n}$ are distinct algebraic numbers, then $e^{a_{1}},e^{a_{2}%
},\ldots,e^{a_{n}}$ are linearly independent complex numbers over the field
$\overline{\mathbb{Q}}$ of all algebraic numbers.

Starting from some basic ideas of Hermite, Lindemann, Hilbert, Hurwitz and
Baker, in this note we provide an easy to understand and self-contained proof
for the Lindemann-Weierstrass Theorem. In an introductory section we have
gathered all the algebraic number theory tools that are necessary to prove the main
theorem. All these auxiliary results are fully proved in a simple and
elementary way, so that the paper can be read even by an undergraduate student.

\end{abstract}
\maketitle

\section*{Introduction}

Trying to prove the transcendence of $e$, C. Hermite~\cite{H} introduced a new
method to approximate the integer powers of $e$ by rational functions with
integer coefficients. His proof is complicated and it seems that it has some
gaps. Following Hermite's main idea (Hermite's Principle), nine years later, F.
Lindemann~\cite{Li}  succeeded to prove the transcendence of $\pi$ by using a
new algebraic improvement. After eleven years, in 1893, D. Hilbert~\cite{Hi}
and A. Hurwitz~\cite{Hu} managed to give simpler proofs for the transcendence
of $e$ and $\pi.$ In 1882 F. Lindemann~\cite{Li1} gives a slight
generalization of his previous result \cite{Li} (Theorem ~\ref{T21} in this
paper). In fact, in this last paper, Lindemann proved the linear independence
of $e^{a_{1}},e^{a_{2}},\ldots,e^{a_{n}}$ (where $a_{1},a_{2},\ldots,a_{n}$
are distinct algebraic numbers) over $\mathbb{Q}$. In 1885 K. Weierstrass
\cite{W} managed to improve and generalize Lindemann's last result even over
$\overline{\mathbb{Q}},$ instead of $\mathbb{Q}$ (Theorem ~\ref{T22} in this
paper). The initial proof given by Weierstrass has been substantially improved
by many mathematicians up to the present day.

In this paper we follow an idea of Baker \cite{B} and we try to make things as
simple as possible. In Section 1 we provide some elementary results from
algebraic number theory that are useful in proving the main results.
In Section~\ref{Section2} we give a simple proof for the transcendence of $\pi$.
Finally, in Section~\ref{Section3}, we supply complete and self-contained proofs for the main results
mentioned above. Lindemann's theorem (see Theorem ~\ref{T21} below) can be
considered as a kind of lemma, since we use it to give a shorter proof for
Theorem ~\ref{T22}, the Lindemann-Weierstrass Theorem (Baker's version).
This note is a natural continuation of our previous paper~\cite{P}.

\section{Some elementary prerequisites in algebraic number theory}\label{Section1}

Some results of this section can also be found in \cite{P}, but here they are
largely improved and completed with new ones.

In the following, $\mathbb{N}$ $=\{0,1,\ldots\}$ denotes the set of natural
numbers, and $\mathbb{Q}$ the set of rational numbers. Let $\mathbb{C}$ denote
the field of complex numbers, which is algebraically closed, that is, any
polynomial $P(x)\in\mathbb{C}[x],$ with $\deg P\geq1,$ has a root in
$\mathbb{C}$ (see~\cite[Chapter VI, 2, Example 5]{L}).
A subfield $K$ in
$\mathbb{C}$ is a non-empty subset of $\mathbb{C}$ which contains $0$ and $1$
such that $K$ is closed to addition, subtraction, multiplication, and
$K^{\ast}=K\smallsetminus\{0\}$ is closed to division. It is not difficult to
see that $\mathbb{Q}$, $\mathbb{R}$, $\mathbb{Q}[\sqrt{2}],$ $\mathbb{Q}[i],$
and so on, are subfields in $\mathbb{C}$. Let $K$ be a subfield of
$\mathbb{C}$. Because $\mathbb{Q}\subset K\subset\mathbb{C}$, we see that $K$
is an extension of $\mathbb{Q}$ and denote it by $K/\mathbb{Q}$. In general,
if $K$ and $L$ are two subfields in $\mathbb{C}$, such that $K\subset L,$ we
say that $L$ is a (field) extension of $K$ and we denote it by $L/K$.

\begin{definition}
\label{D11} Let $K\subset\mathbb{C}$ be a subfield in $\mathbb{C}$. An element
$\alpha\in\mathbb{C}$ is said to be an algebraic number over $K$ if it is a
root of a polynomial $P\in K[x]$ of degree greater than zero. If
$K=\mathbb{Q}$, we simply say that $\alpha$ is an algebraic number.
\end{definition}

\begin{remark}
\label{R77} Let $\alpha\in\mathbb{C}$ be an algebraic number over $K$, where
$K$ is a subfield of $\mathbb{C}$, and let
\[
P(x)=a_{n}x^{n}+a_{n-1}x^{n-1}+\cdots+a_{0}\in K[x],\text{ }a_{n}\neq0,
\]
such that $P(\alpha)=0.$ Since $\frac{1}{a_{n}}P(\alpha)=0$, we can always
assume that $P$ is a monic polynomial. A monic polynomial $P\in K[x]$ with
$P(\alpha)=0$ and of minimal degree with these last two properties is called
the \textit{minimal polynomial of }$\alpha$\textit{ over }$K$ and we denote it
by $f_{\alpha,K}.$ The other roots $\alpha_{2},\ldots,\alpha_{m}$ of
$f_{\alpha,K}$ $(m=\deg_{K}f_{\alpha,K}=\deg_{K}\alpha)$ are called the
conjugates of $\alpha$ $($over $K)$. If $K=$ $\mathbb{Q}$, we simply write
$f_{\alpha}$ for the minimal polynomial of $\alpha$ over $\mathbb{Q}$. We see
that $f_{\alpha,K}(0)\neq0$ if $\deg f_{\alpha,K}\geq2$, $f_{\alpha,K}$ is
irreducible in $K[x],$ $f_{\alpha,K}=$ $f_{\beta,K}$ for any other root
$\beta$ of $f_{\alpha,K}$, this last one is unique, and it has simple roots
$($otherwise $f_{\alpha,K}^{\prime}(\beta)=f_{\beta,K}^{\prime}(\beta)=0$ and
$\deg f_{\beta,K}^{\prime}<\deg f_{\beta,K},$ etc.$)$. If $g\in K[x]$ such
that $g(\alpha)=0$ then, applying the Euclidean division algorithm to $g$ and
$f_{\alpha,K},$ $g=f_{\alpha,K}h+r,$ where $h,r\in K[x]$ and $\deg r<\deg
f_{\alpha,K}$. The minimality of $f_{\alpha,K}$ implies $r=0$ in $K[x]$, that
is, $g$ is divisible by $f_{\alpha,K}.$
\end{remark}

Let $\alpha\in\mathbb{C}$ be a complex number and let $K$ be a subfield in
$\mathbb{C}$. We denote by ~$K[\alpha]$ the least subring of $\mathbb{C}$
generated by $K$ and $\alpha$, and $K(\alpha)$ denotes the least subfield of
$\mathbb{C}$ generated by $K$ and $\alpha$.

\begin{lemma}[{\!\!\cite[Lemma 6.3]{P}}]\label{L11} 
With the above notation and assumptions,
the following statements are equivalent:

\begin{itemize}
\item[(i)] $\alpha$ is an algebraic number over $K$.

\item[(ii)] The vector space $K[\alpha]$ has a finite dimension over $K$. In
this last case $\dim_{K}K[\alpha]=\deg_{K}\alpha$.

\item[(iii)] $K[\alpha]=K(\alpha)$ $($in this last case we prefer to write
$K[\alpha]$ instead of $K(\alpha))$.
\end{itemize}
\end{lemma}

\begin{proof}
We begin by proving that $(i)$ implies $(ii)$. Let%
\begin{equation}
f_{\alpha,K}(x)=x^{n}+a_{n-1}x^{n-1}+\cdots+a_{0}\in K[x] \label{FF1}%
\end{equation}
be the minimal polynomial of $\alpha$ over $K.$ Thus%
\[
\alpha^{n}=-a_{n-1}\alpha^{n-1}-\cdots-a_{0}\in K[\alpha].
\]
Thus,%
\[
\alpha^{n+1}=-a_{n-1}[-a_{n-1}\alpha^{n-1}-\cdots-a_{0}]-a_{n-2}\alpha
^{n-1}-\cdots-a_{0}\alpha,
\]
which is of the form%
\[
\alpha^{n+1}=b_{n-1}^{(1)}\alpha^{n-1}+\cdots+b_{0}^{(1)}\in K[\alpha].
\]
In general, we see that%
\[
\alpha^{n+k}=b_{n-1}^{(k)}\alpha^{n-1}+\cdots+b_{0}^{(k)}\in K[\alpha]
\]
for $k=0,1,\ldots$. Here $b_{n-1}^{(0)}=-a_{n-1},\ldots,b_{0}^{(0)}=-a_{0}$
and any $b_{j}^{(k)}\in K$ for ~$j=0,1,\ldots,n-1$ and $k=0,1,\ldots$.

Thus, $\{1,\alpha,\ldots,\alpha^{n-1}\}$ is a generating system for
$K[\alpha]$ over $K$. Let us take a null linear combination of $1,\alpha
,\ldots,\alpha^{n-1}$ over $K$,%
\[
c_{0}+c_{1}\alpha+\cdots+c_{n-1}\alpha^{n-1}=0,
\]
where $c_{0},\ldots,c_{n-1}\in K.$ If not all these elements are zero, then
the polynomial%
\[
g(x)=c_{0}+c_{1}x+\cdots+c_{n-1}x^{n-1}\in K[x],
\]
with $1\leq\deg_{K}g<\deg_{K}f_{\alpha,K},$ is equal to zero when $x=\alpha$,
a contradiction relative to the minimality of the degree of $f_{\alpha,K}$.
Thus, $\{1,\alpha,\ldots,\alpha^{n-1}\}$ is a basis of $K[\alpha]$ over $K$
and
\[
\dim_{K}K[\alpha]=n=\deg_{K}\alpha.
\]

Now we prove that $(ii)$ implies $(iii)$. Let $0\neq g(\alpha)=c_{0}%
+c_{1}\alpha+\cdots+c_{n-1}\alpha^{n-1}$ be in $K[\alpha]$, the least subring
in $\mathbb{C}$ generated by $K$ and $\alpha$. Since $\deg_{K}g(x)<\deg
_{K}f_{\alpha,K}$ and $f_{\alpha,K}$ is irreducible, we see that the greatest
common divisor of $g$ and $f_{\alpha,K}$ is ~$1$. We repeat the Euclidean
division algorithm for $f_{\alpha,K}$ and $g$ in $K[x]$ and find two
polynomials $u,v\in K[x]$ such that $f_{\alpha,K}(x)u(x)+g(x)v(x)=1$. If we
make $x=\alpha$ in this last equality, we find $g(\alpha)v(\alpha)=1$, that
is, $\left[  g(\alpha)\right]  ^{-1}=v(\alpha)\in K[\alpha]$. Thus ~$K[\alpha]$
is a subfield in $\mathbb{C}$, and consequently $K[\alpha]=K(\alpha)$.

Now we prove that $(iii)$ implies $(i)$. If $\alpha=0,$ we have nothing to
prove. We assume that $\alpha\neq0$ and $K[\alpha]=K(\alpha)$. Thus,%
\[
\frac{1}{\alpha}=d_{0}+d_{1}\alpha+\cdots+d_{n-1}\alpha^{n-1}\in K[\alpha].
\]
Therefore, $\alpha$ is a root of the polynomial
\[
h(x)=d_{n-1}x^{n}+d_{n-2}x^{n-1}+\cdots+d_{0}x-1\in K[x],
\]
that is, $\alpha$ is an algebraic number over $K,$ and the proof is finished.
\end{proof}

\begin{lemma}
\label{LA1} Let $\mathbb{Q}\subset K\subset L\subset\mathbb{C}$ be a tower of
subfields in $\mathbb{C}$ with $K:\mathbb{Q}=n$ and $L:K=m$. Let $u_{1}%
,u_{2},\ldots,u_{n}$ be a basis of $K$ over $\mathbb{Q}$ and let $v_{1}%
,v_{2},\ldots,v_{m}$ be a basis of $L$ over $K.$ Then $\{u_{i}v_{j}\}$,
$i=1,2,\ldots,n,$ $j=1,2,\ldots,m$, is a basis of $L$ over ~$\mathbb{Q}$. In
particular, $\dim_{\mathbb{Q}}L=nm<\infty$.
\end{lemma}

\begin{proof}
Let $a$ be an element in $L.$ Thus,%
\[
a=\sum_{j=1}^{m}a_{j}v_{j},
\]
where $a_{j}\in K$ for $j=1,2,\ldots,m$. For each $j=1,2,\ldots,m$, we can
write%
\[
a_{j}=\sum_{i=1}^{n}a_{ji}u_{i},
\]
where $a_{ji}\in\mathbb{Q}$ for $i=1,2,\ldots,n$. Thus,%
\[
a=\sum_{j=1}^{m}\sum_{i=1}^{n}a_{ji}u_{i}v_{j},
\]
that is, $\{u_{i}v_{j}\}$, $i=1,2,\ldots,n$, $j=1,2,\ldots,m$, is a generating
system for $L$ over $\mathbb{Q}$. It is also a linear independent set over
$\mathbb{Q}$. Indeed, if%
\[
\sum_{j=1}^{m}\sum_{i=1}^{n}a_{ji}u_{i}v_{j}=\sum_{j=1}^{m}\left(  \sum
_{i=1}^{n}a_{ji}u_{i}\right)  v_{j}=0,
\]
then $\sum_{i=1}^{n}a_{ji}u_{i}=0$ for any $j=1,2,\ldots,m$, because
$\{v_{1},v_{2},\ldots,v_{m}\}$ is a basis of ~$L$ over $K$. Since
$\{u_{1},u_{2},\ldots,u_{n}\}$ is a basis of $K$ over $\mathbb{Q}$, we see
that $a_{ji}=0$ for any $i=1,2,\ldots,n$ and $j=1,2,\ldots,m$. Therefore
$\{u_{i}v_{j}\}$, $i=1,2,\ldots,n$, $j=1,2,\ldots,m$, is a basis of $L$ over
$\mathbb{Q}$, and the lemma is proved.
\end{proof}

In what follows we denote by $\overline{\mathbb{Q}}$ the subset of all
algebraic numbers (over $\mathbb{Q}$) in ~$\mathbb{C}$.

\begin{lemma}[{\!\!\cite[Corollary 6.4]{P}}]\label{L21} 
$\overline{\mathbb{Q}}$ is a subfield of $\mathbb{C}$.
\end{lemma}

\begin{proof}
Let $\alpha$, $\beta\in\mathbb{C}$ be two algebraic numbers (over $\mathbb{Q}%
$) and let $K=\mathbb{Q}[\alpha]$ be the subfield of $\mathbb{C}$ generated by
$\alpha$. Since $\beta$ is an algebraic number (over $\mathbb{Q}$), it is also
an algebraic number over $K$. Thus, $K[\beta]:K<\infty$ (Lemma ~\ref{L11}),
and the tower of finite extensions%
\[
\mathbb{Q}\subset\mathbb{Q}[\alpha]=K\subset K[\beta]=\mathbb{Q}[\alpha
,\beta]
\]
says that $\mathbb{Q}[\alpha,\beta]:\mathbb{Q}<\infty$ (Lemma ~\ref{LA1}).
Therefore, we see that for any $\gamma\in\mathbb{Q}[\alpha,\beta]$,
$\gamma\neq0$, the set $\{1,\gamma,\gamma^{2},\dots\}$ is linear dependent over
$\mathbb{Q}$. This means that there is a null nontrivial linear combination%
\[
a_{0}+a_{1}\gamma+\cdots+a_{n}\gamma^{n}=0,
\]
with $a_{0},a_{1},\dots,a_{n}\in\mathbb{Q}$, $a_{n}$, $a_{0}\neq0$. Thus,
$\gamma$ is a root of the nontrivial polynomial%
\[
P(x)=a_{0}+a_{1}x+\cdots+a_{n}x^{n}\in\mathbb{Q}[x],
\]
that is, it is an algebraic number. In particular, $\gamma=\alpha\pm
\beta,\alpha\beta\in\mathbb{Q}[\alpha,\beta]$ are algebraic numbers. Since for
any nonzero algebraic number $\delta$, $\mathbb{Q}[\delta]$ is a subfield in
$\mathbb{C}$ (Lemma \ref{L11}), we see that $1/\delta\in\mathbb{Q}[\delta]$,
that is, $1/\delta$ is also an algebraic number. Thus $\overline{\mathbb{Q}}$
is a subfield of $\mathbb{C}$ and the proof is complete.
\end{proof}

\begin{remark}
\label{R11} Let $K$ be a subfield in $\mathbb{C}$ and let $\alpha_{1}%
,\alpha_{2},\ldots,\alpha_{n}$ be algebraic numbers over $K$. Then
$K[\alpha_{1},\alpha_{2},\ldots,\alpha_{n}]$ is a subfield of $\mathbb{C}$ and
$K[\alpha_{1},\alpha_{2},\ldots,\alpha_{n}]:K<\infty$. In particular, any
element of $K[\alpha_{1},\alpha_{2},\ldots,\alpha_{n}]$ is an algebraic number
over $K$. Indeed, for $n=1$ we see this from Lemma ~\ref{L11}. We assume that
we have proved the statement for $k=1,2,\ldots,n-1$. Since
\[
K[\alpha_{1},\alpha_{2},\ldots,\alpha_{n}]=K[\alpha_{1},\alpha_{2}%
,\ldots,\alpha_{n-1}][\alpha_{n}]
\]
and, because $\alpha_{n}$ is also algebraic over $K[\alpha_{1},\alpha
_{2},\ldots,\alpha_{n-1}]$, we see that
\[
K[\alpha_{1},\alpha_{2},\ldots,\alpha_{n-1}][\alpha_{n}]:K[\alpha_{1}%
,\alpha_{2},\ldots,\alpha_{n-1}]<\infty.
\]
But, the tower of finite extensions%
\[
K\subset K[\alpha_{1},\alpha_{2},\ldots,\alpha_{n-1}]\subset K[\alpha
_{1},\alpha_{2},\ldots,\alpha_{n}]
\]
and the induction assumption say that $K[\alpha_{1},\alpha_{2},\ldots
,\alpha_{n}]:K<\infty$. Let $\gamma$ be in $K[\alpha_{1},\alpha_{2}%
,\ldots,\alpha_{n}]$. Since $K[\gamma]\subset K[\alpha_{1},\alpha_{2}%
,\ldots,\alpha_{n}]$, we see that $\dim_{K}K[\gamma]<\infty$ and Lemma
~\ref{L11} says that $\gamma$ is algebraic over $K$. Therefore, the above
statement is proved.
\end{remark}

\begin{definition}
\label{D21} Let $K\subset L\subset\mathbb{C}$ be a tower of subfields in
$\mathbb{C}$. A field morphism $\sigma:L\rightarrow\mathbb{C}$, this meaning a
mapping $\sigma$ with the following properties
\[
\sigma(\alpha\pm\beta)=\sigma(\alpha)\pm\sigma(\beta),\sigma(\alpha
\beta)=\sigma(\alpha)\sigma(\beta),\sigma(1)=1,
\]
is said to be a $K$-embedding of $L$ in $\mathbb{C}$ if $\sigma(\gamma
)=\gamma$ for any $\gamma\in K$. If $K=\mathbb{Q}$ a $\mathbb{Q}$-embedding of
$L$ in $\mathbb{C}$ is simply a field morphism from $L$ to $\mathbb{C}$. Such
a field morphism from $L$ to $\mathbb{C}$ is simply called an embedding of $L$
in $\mathbb{C}$.
\end{definition}

\begin{remark}
\label{R701} It is not difficult to see that if $K\subset L\subset
\overline{\mathbb{Q}}$ is a tower of subfields in $\mathbb{C}$, and $\sigma$
is a $K$-embedding of $L$ in $\mathbb{C}$, then $\sigma(L)\subset
\overline{\mathbb{Q}}$, because, through $K$-embeddings the algebraic numbers
over $K$ are also transformed into algebraic numbers over $K$. Moreover, if
$L$ is a finite \textit{normal extension} of $K$, that is, if it is generated
over $K$ by all the roots of a finite set of polynomials with coefficients in
$K$, then $\sigma(L)=L$. Indeed, in general, the normality of $L$ implies
$\sigma(L)\subset L$, because a root of a polynomial $P\in K[x]$ is
transformed into a root of the same polynomial. Since we have 
$L:K=\sigma(L):K<\infty
$, we conclude that $\sigma(L)=L$. In this last case, that is, when ~$L$ is a
normal finite extension of $K$, the set of all $K$-embeddings of $L$ $($in
$L$, because $\sigma(L)=L)$ is a group relative to the usual composition law
of automorphisms. It is called the \textit{Galois group of the extension }%
$L/K$, and it is denoted by $Gal(L/K)$. The word embedding comes from the fact
that any field morphism $\mu:L\rightarrow\mathbb{C}$ is one-to-one
$($injective$)$. If $L/K$ is a finite normal extension, then any 
$K$-embedding~$\sigma$ of $L$ is also an onto $($surjective$)$ mapping, that is, $\sigma$ is
an automorphism.
\end{remark}

Let $\beta$ be an algebraic number and $\sigma:\mathbb{Q}[\beta]\rightarrow
\mathbb{C}$ be an embedding of $\mathbb{Q}[\beta]$ in $\mathbb{C}$. We see
that $\sigma(\beta)$ is also a root of $f_{\beta}\in\mathbb{Q}[x]$, where
$f_{\beta}$ is the minimal polynomial of $\beta$ (over $\mathbb{Q}$).

Let $\mathbb{Q}\subset K\subset L\subset\mathbb{C}$ be a tower of subfields in
$\mathbb{C}$ and let $\sigma:K\rightarrow\mathbb{C}$ be an embedding of $K$
into $\mathbb{C}$. We say that an embedding $\mu:L\rightarrow\mathbb{C}$ of
$L$ in $\mathbb{C}$ \textit{extends} $\sigma$ to $L$ (or that $\mu$ is an
\textit{extension of }$\sigma$\textit{ to }$L$) if $\mu(\alpha)=\sigma
(\alpha)$ for any $\alpha\in K.$

\begin{lemma}[{\!\!\cite[Lemma 6.9]{P}}]\label{L31} 
Let $K$ be a subfield of $\mathbb{C}$
such that $K/\mathbb{Q}$ is a finite normal extension and let $\alpha
\in\mathbb{C}$ be an algebraic element over $K$. Let $\sigma$ be a fixed
embedding of $K$ into $\mathbb{C}$. Then, the number of embeddings
$\mu:K[\alpha]\rightarrow\mathbb{C}$, that extend $\sigma$ to $K[\alpha]$, is
equal to $K[\alpha]:K=\deg_{K}\alpha=\deg f_{\alpha,K}$.
\end{lemma}

\begin{proof}
Any element of $K[\alpha]$ is of the form,
\[
g(\alpha)=a_{0}+a_{1}\alpha+\cdots+a_{m}\alpha^{m},
\]
where $g\in K[x]$. We write
\[
g^{\sigma}(x)=\sigma(a_{0})+\sigma(a_{1})x+\cdots+\sigma(a_{m})x^{m}\in
\sigma(K)[x].
\]
It is not difficult to see that the mapping $g\rightarrow g^{\sigma}$ is a
ring isomorphism from $K[x]$ to ~$\sigma(K)[x]$. Let $f=f_{\alpha,K}$ be the
minimal polynomial of $\alpha$ over $K$ and let $\mu:K[\alpha]\rightarrow
\mathbb{C}$ be an embedding of $K[\alpha]$ in $\mathbb{C}$ that extends
$\sigma$ to $K[\alpha]$. Since $f(\alpha)=0$, we see that $\mu(\alpha)$ is a
root of the irreducible polynomial $f^{\sigma}\in\sigma(K)[x]$. Conversely,
any root $\beta$ of the irreducible polynomial $f^{\sigma}\in\sigma(K)[x]$
gives rise to an embedding $\mu_{\beta}:K[\alpha]\rightarrow\mathbb{C}$, which
extend $\sigma$ to $K[\alpha]$. Indeed, let us define%
\[
\mu_{\beta}(g(\alpha))\overset{def}{=}g^{\sigma}(\beta).
\]
This is a specialization of the ring morphism $g\rightarrow g^{\sigma}$. That
is why sums go into sums and products go into products. All that remains is to
prove the well definition. It is enough to show that if $g(\alpha)=0$, then
$g^{\sigma}(\beta)$ is also zero. Let us assume that $g(\alpha)=0$. Then $g$
is divisible by $f$, that is, $g=fh$ in $K[x]$. Hence $g^{\sigma}=f^{\sigma
}h^{\sigma}$ in $\sigma(K)[x]$, and consequently $g^{\sigma}(\beta)=f^{\sigma
}(\beta)\cdot h^{\sigma}(\beta)=0$, because $\beta$ is a root of $f^{\sigma}.$
Since $f^{\sigma}$ is irreducible over $\sigma(K)$ and since $\deg_{K}f$
$=\deg_{\sigma(K)}f^{\sigma}$, we see that the number of embeddings of
$K[\alpha]$ which extends $\sigma$ to $K[\alpha]$ is equal to $K[\alpha]:K$,
and the proof is finished.
\end{proof}

\begin{lemma}
\label{L41} Let $\mathbb{Q}\subset L\subset\mathbb{C}$ be a tower of subfields
in $\mathbb{C}$ such that $L:\mathbb{Q}<\infty$. Then, the number of
embeddings of $L$ in $\mathbb{C}$ is equal to the dimension $L:\mathbb{Q}$ of
$L$ as a $\mathbb{Q}$-vector space. In particular, if $L$ is a finite normal
extension of $\mathbb{Q}$, then the Galois group $Gal(L/\mathbb{Q})$ has
exactly $L:\mathbb{Q}$ elements.
\end{lemma}

\begin{proof}
Since $L:\mathbb{Q}<\infty$, there are some algebraic numbers $\alpha
_{1},\alpha_{2},\ldots,\alpha_{m}\in L$ such that $L=\mathbb{Q}[\alpha
_{1},\alpha_{2},\ldots,\alpha_{m}]$, the least subfield of $\mathbb{C}$
generated by $\alpha_{1},\alpha_{2},\ldots,\alpha_{m}$. Moreover, we can
assume that%
\[
\mathbb{Q}\subset\mathbb{Q}[\alpha_{1}]\subset\mathbb{Q}[\alpha_{1}%
][\alpha_{2}]\subset\cdots\subset\mathbb{Q}[\alpha_{1}][\alpha_{2}%
]\cdots\lbrack\alpha_{m}]=L,
\]
where the inclusions are strict. Now, we apply Lemma ~\ref{L31} to each simple
extension
\[
K_{i}=\mathbb{Q}[\alpha_{1},\alpha_{2},\ldots,\alpha_{i}]\subset K_{i+1}%
=K_{i}[\alpha_{i+1}],i=0,1,\ldots,m-1,K_{0}=\mathbb{Q},
\]
and find that any embedding $\lambda$ of $K_{i}$ in $\mathbb{C}$ can be
extended to $K_{i+1}:K_{i}$ embeddings~$\mu$ of $K_{i+1}$ in $\mathbb{C}$
(Lemma \ref{L31}). Since%
\[
L:\mathbb{Q}=\prod_{i=0}^{m-1}[K_{i+1}:K_{i}]
\]
(Lemma ~\ref{LA1}), we obtain the first statement of the lemma.

To prove the last statement, it is enough to see that any embedding $\mu$ of
$L$ in ~$\mathbb{C}$ has values in $L$, because $L$ is a normal extension, and
consequently, $\mu$ permutes the roots of an irreducible polynomial. Moreover,
since $L:\mathbb{Q}$ $=\mu(L):\mathbb{Q}$, we see that $\mu$ is also onto on
$L$. Thus $\mu$ is in fact an automorphism of $L$, that is, it is an element
of $Gal(L/\mathbb{Q})$ (Remark ~\ref{R701}), and the proof is complete.
\end{proof}

\begin{corollary}
\label{C91} Let $K$ be a subfield of $\mathbb{C}$ with $K:\mathbb{Q}$
$<\infty$ and let $\alpha\in K$ be such that for any embedding $\sigma
:K\rightarrow\mathbb{C}$, $\sigma(\alpha)=\alpha$. Then $\alpha\in\mathbb{Q}$.
\end{corollary}

\begin{proof}
Assume that $\alpha\notin\mathbb{Q}$, so that $\mathbb{Q}[\alpha]:\mathbb{Q}$
$>1$. Lemma ~\ref{L41} says that the number of the embeddings of
$\mathbb{Q}[\alpha]$ in $\mathbb{C}$ is greater than $1$. Take $\mu:$
$\mathbb{Q}[\alpha]\rightarrow\mathbb{C}$, such that $\mu(\alpha)\neq\alpha$
($\mu(\alpha)$ is a conjugate of $\alpha,$ distinct of $\alpha$). Now, we
apply Lemma ~\ref{L41} and we find an embedding $\sigma$ of $K$ into
$\mathbb{C}$ that extends $\mu$ to $K$. Since $\sigma(\alpha)=\mu(\alpha
)\neq\alpha,$ we obtain a contradiction. Hence $\alpha\in\mathbb{Q}$, and the
proof is finished.
\end{proof}

\begin{definition}\label{D31} Let $K$ be a subfield of $\mathbb{C}$ such that $K:\mathbb{Q}=n$,
and let $\sigma_{1},\sigma_{2},\dots,\sigma_{n}$ be all the embeddings of $K$ in
$\mathbb{C}$ $($Lemma \ref{L41}$)$. For any $\alpha\in K$ we define%
\[
N_{K/\mathbb{Q}}(\alpha)=\sigma_{1}(\alpha)\cdot\sigma_{2}(\alpha)\cdots
\sigma_{n}(\alpha),
\]
and call it the norm of $\alpha$ relative to $K$. If $K=\mathbb{Q}[\alpha]$,
then $N_{\mathbb{Q}[\alpha]/\mathbb{Q}}$ $(\alpha)$ is simply called the norm
of $\alpha$ and we denote it by $N(\alpha)$. In this last case, we know that
$N(\alpha)\in\mathbb{Q}$ $($it is the product of all the roots of $f_{\alpha
},$ that is, it is equal to $(-1)^{n}f_{\alpha}(0)$, according to Vi\`{e}te's
formulas$)$.
\end{definition}

Using only the multiplicative property of embeddings, we can prove the
following lemma.

\begin{lemma}[{\!\!\cite[Lemma 6.21]{P}}]\label{L50} 
Let $K$ be a finite extension of
$\mathbb{Q}$, and let $\alpha,\beta\in K$. Then,%
\begin{equation}
N_{K/\mathbb{Q}}(\alpha\beta)=N_{K/\mathbb{Q}}(\alpha)N_{K/\mathbb{Q}}(\beta).
\label{102}%
\end{equation}

\end{lemma}

We also have the following useful result.

\begin{lemma}[{\!\!\cite[Lemma 6.20]{P}}]\label{L51} 
With the above notation and definition,
we have%
\[
N_{K/\mathbb{Q}}(\alpha)=\left[  (-1)^{m}f_{\alpha}(0)\right]  ^{k},
\]
where $m=\deg_{\mathbb{Q}}\alpha$, and $k=K:\mathbb{Q}[\alpha]$. In
particular, $N_{K/\mathbb{Q}}(\alpha)\in\mathbb{Q}$.
\end{lemma}

\begin{proof}
For any fixed embedding $\mu$ of $\mathbb{Q}[\alpha]$, there exist exactly $k$
embeddings $\sigma$ of ~$K$ which extends this $\mu$ (Lemma ~\ref{L31}). Thus,
the latter can be grouped in such a way that in each group of $k$, their
restrictions to $\mathbb{Q}[\alpha]$ are one and the same embedding~$\mu$ of
$\mathbb{Q}[\alpha]$. Then, we use a remark made in Definition ~\ref{D31} and,
consequently, the lemma is proved.
\end{proof}

\begin{corollary}[{\!\!\cite[Corollary 6.22]{P}}]\label{C21} 
Let $q$ be in $\mathbb{Q}$ and let
$K$ be a subfield of $\mathbb{C}$, such that $K:\mathbb{Q}$ $=n$. Let $\alpha$
be in $K$. Then,%
\[
N_{K/\mathbb{Q}}(q\alpha)=q^{n}N_{K/\mathbb{Q}}(\alpha).
\]

\end{corollary}

\begin{proof}
Since $q\in\mathbb{Q}$, $\mathbb{Q}[q]=\mathbb{Q}$, thus we simply apply
formula (\ref{102}) and find:%
\[
N_{K/\mathbb{Q}}(q\alpha)=N_{K/\mathbb{Q}}(q)N_{K/\mathbb{Q}}(\alpha
)=[(-1)(-q)]^{n}N_{K/\mathbb{Q}}(\alpha)=q^{n}N_{K/\mathbb{Q}}(\alpha).
\]

\end{proof}

\begin{definition}
\label{D41} A complex number $\alpha$ is said to be an algebraic integer if it
is a root of a monic polynomial $P\in\mathbb{Z}[x]$ with $\deg P\geq1$.
\end{definition}

\begin{remark}[{\!\!\cite[Lemma 6.25]{P}}]\label{R51} 
For any algebraic number $\alpha$ there
is a positive integer $d$ such that $d\alpha$ is an algebraic integer. Indeed,
if%
\[
\alpha^{n}+\frac{b_{n-1}}{d}\alpha^{n-1}+\frac{b_{n-2}}{d}\alpha^{n-2}%
+\cdots+\frac{b_{1}}{d}\alpha+\frac{b_{0}}{d}=0,
\]
where $b_{0},b_{1},\ldots,b_{n-1}$, $d\in\mathbb{Z}$, $d>0$, $n\geq1$, then
$d\alpha$ is a root of the following monic polynomial $P$ with integer
coefficients,%
\[
P(x)=x^{n}+b_{n-1}x^{n-1}+db_{n-2}x^{n-2}+\cdots+d^{n-1}b_{0}.
\]

\end{remark}

\begin{remark}
\label{R41} Any algebraic integer $\alpha\in\mathbb{Q}$ is an integer, that
is, $\alpha\in\mathbb{Z}$. Indeed, if $\alpha=m/n$ with $(m,n)=1,$
$m,n\in\mathbb{Z}$, $n>0,$ then there exists a relation of the following type,%
\[
\left(  \frac{m}{n}\right)  ^{k}+b_{k-1}\left(  \frac{m}{n}\right)
^{k-1}+\cdots+b_{0}=0,
\]
where $k\in\mathbb{N}^{\ast}$and $b_{0},b_{1},\dots,b_{k-1}\in\mathbb{Z}$. We
multiply this last equality by $n^{k}$ and we find,%
\[
m^{k}+b_{k-1}m^{k-1}n+\cdots+b_{0}n^{k}=0.
\]
Thus, $n$ is a divisor of $m^{k}$. Since $(m,n)=1$, we see that $n=1$, that
is, $\alpha=m\in\mathbb{Z}$.
\end{remark}

\begin{remark}
\label{R61} Let $\alpha$ be an algebraic integer, $\alpha\in K$, a finite
extension of $\mathbb{Q}$ $(K\subset\overline{\mathbb{Q}})$. Let $\sigma$ be
an embedding of $K$ in $\mathbb{C}$. Then, $\sigma(\alpha)$ is also an
algebraic integer $(0=\sigma(P(\alpha))=P(\sigma(\alpha)))$. Hence,
$f_{\alpha}\in\mathbb{Z}[x]$ $($use Vi\`{e}te formulas and Remark
~\ref{R41}$)$. Consequently, $N_{K}(\alpha)\in\mathbb{Z}$.
\end{remark}

\begin{lemma}[{\!\!\cite[Lemma 6.26]{P}}]\label{L61} 
The subset $\mathbb{A}$ of all
algebraic integers in $\mathbb{C}$ is a subring of $\overline{\mathbb{Q}}$.
\end{lemma}

\begin{proof}
Let $\alpha\neq0$, $\beta\neq0$ be two nonzero algebraic integers and let
$f_{\alpha}$, $f_{\beta}\in\mathbb{Z}[x]$ (Remark ~\ref{R61}) be the minimal
polynomials of $\alpha$ and $\beta$ respectively. Let $n=\deg f_{\alpha}$ and
$m=\deg f_{\beta}$ be their degrees. Thus, there exist $a_{0},a_{1}%
,\ldots,a_{n-1}$, $b_{0},b_{1},\ldots,b_{m-1}\in\mathbb{Z}$ such that%
\[
\alpha^{n}=-a_{0}-a_{1}\alpha-\cdots-a_{n-1}\alpha^{n-1},
\]
and%
\[
\beta^{m}=-b_{0}-b_{1}\beta-\cdots-b_{m-1}\beta^{m-1}.
\]
Thus, any element $s\in\mathbb{Z}[\alpha,\beta]=\mathbb{Z}[\alpha][\beta]$ can
be written as%
\[
s=\sum_{i=0}^{n-1}\sum_{j=0}^{m-1}c_{ij}\alpha^{i}\beta^{j},\text{ }c_{ij}%
\in\mathbb{Z}.
\]
We denote $\omega_{1},\omega_{2},\ldots,\omega_{k}$, $k=nm$, the elements of
the generating set $\{\alpha^{i}\beta^{j}\},$ $i=0,1,\ldots,n-1$,
$j=0,1,\ldots,m-1$ of $\mathbb{Z}[\alpha,\beta]$ (over $\mathbb{Z}$). For any
$\gamma\in\mathbb{Z}[\alpha,\beta]$ we can write,%
\begin{equation}
\gamma\omega_{i}=\sum_{j=1}^{k}a_{ij}\omega_{j}, \label{107}%
\end{equation}
where $a_{ij}\in\mathbb{Z}$ for any $i\in\{1,2,\ldots,k\}$. Let $B=\gamma
I_{k}-A$, where $I_{k}$ is the $k\times k$ identity matrix and $A$ is the
matrix $(a_{ij})$ which belongs to $\mathcal{M}_{k}(\mathbb{Z})$, the ring of
all $k\times k$ matrices with entries in $\mathbb{Z}.$ Thus, formula
(\ref{107}) can also be written as
\[
B 
\begin{pmatrix}
    \omega_{1}\\
    \omega_{2}\\
    \vdots\\
    \omega_{k}%
\end{pmatrix}
=
\begin{pmatrix}
    0\\
    0\\
    \vdots\\
    0    
\end{pmatrix},
\]
which is a homogeneous system with a nontrivial solution $\omega_{1}%
,\omega_{2},\ldots,\omega_{k}$. Therefore, $\det B=0$, that is, $\gamma$ is a
root of the monic polynomial $P(x)=\det$ $(xI_{k}-A)\in\mathbb{Z}[x]$. Now,
for $\gamma=\alpha\pm\beta$ or $\gamma=\alpha\beta$, we obtain the statement
of the lemma, which concludes our proof.
\end{proof}

\begin{definition}
\label{D51} We say that an algebraic integer $\alpha\in\mathbb{A}$ is
divisible by a nonzero integer $n\in\mathbb{Z}$ if there exists another
algebraic integer $\beta\in\mathbb{A}$, such that $\alpha=n\beta$.
\end{definition}

\begin{lemma}[{\!\!\cite[Lemma 6.29]{P}}]\label{L71} 
For any nonzero algebraic integer
$\alpha\in\mathbb{A}$, there exist only a finite number of prime numbers $p$
such that $\alpha$ is divisible by $p$.
\end{lemma}

\begin{proof}
Suppose that $\alpha=p\beta$, $\beta\in\mathbb{A}$ where $p$ is a prime
number. From Corollary ~\ref{C21} we get%
\[
N_{\mathbb{Q}[\beta]/\mathbb{Q}}(\alpha)=p^{n}N_{\mathbb{Q}[\beta]/\mathbb{Q}%
}(\beta),
\]
where $n=$ $\mathbb{Q}[\beta]:\mathbb{Q}$. Since $N_{\mathbb{Q}[\beta]}(\alpha)$ and
$N_{\mathbb{Q}[\beta]}(\beta)$ are integers (Remark ~\ref{R61}) and since
$N_{\mathbb{Q}[\beta]/\mathbb{Q}}(\alpha)$ is a fixed nonzero integer, it
cannot have an infinite number of prime divisors. Hence, the proof is complete.
\end{proof}

\section{ \texorpdfstring{\bigskip$\pi$}{pi} is a transcendental number}\label{Section2}

The first step to prove the Lindemann-Weierstrass theorem is to prove that
$\pi$ itself is a transcendental number, a result obtained in 1882 by
Lindemann \cite{Li}. For this we need an elementary auxiliary result.

\begin{lemma}
\label{L99} Let $\mathcal{M}=\{\gamma_{1},\gamma_{2},\ldots,\gamma_{n}\}$ be a
finite set of algebraic numbers $($over ~$\mathbb{Q}).$ Then, $\sigma
(\mathcal{M})=\mathcal{M}$ for any embedding $\sigma$ of $\mathbb{Q}%
[\gamma_{1},\gamma_{2},\ldots,\gamma_{n}]$ in $\mathbb{C}$, if and only if
$\mathcal{M}$ is the set of roots of a polynomial $P$ with rational coefficients.
\end{lemma}

\begin{proof}
It is not difficult to see that if $\mathcal{M}$ is the set of roots of a
polynomial $P(x)\in\mathbb{Q}[x],$ then any embedding $\sigma$ of the field
$\mathbb{Q}[\gamma_{1},\gamma_{2},\ldots,\gamma_{n}]$ in $\mathbb{C}$ permutes
these roots, so that $\sigma(\mathcal{M})=\mathcal{M}$.

Conversely, let us assume that $\sigma(\mathcal{M})=\mathcal{M}$ for any
embedding $\sigma$ of the field $\mathbb{Q}[\gamma_{1},\gamma_{2}%
,\ldots,\gamma_{n}]$ in $\mathbb{C}$. Let%
\[
P(x)=(x-\gamma_{1})(x-\gamma_{2})\cdots(x-\gamma_{n})=x^{n}-s_{1}%
x^{n-1}+\cdots+(-1)^{n}s_{n},
\]
where
\[
s_{1}=\sum_{j=1}^{n}\gamma_{i},\text{ }s_{2}=\sum_{i,j=1,i<j}^{n}\gamma
_{i}\gamma_{j},\dots,s_{n}=\gamma_{1}\gamma_{2}\cdots\gamma_{n}%
\]
are the fundamental symmetric polynomial in $\gamma_{1},\gamma_{2}%
,\ldots,\gamma_{n}$. Since $\sigma(\mathcal{M})=\mathcal{M}$, we see that
$\sigma(s_{j})=s_{j}$ for any $j=1,2,\ldots,n$ and for any embedding $\sigma$
of the field $\mathbb{Q}[\gamma_{1},\gamma_{2},\ldots,\gamma_{n}]$ in
$\mathbb{C}$. Consequently, from Corollary \ref{C91}, we find that $s_{j}%
\in\mathbb{Q}$ for $j=1,2,\ldots,n$, meaning that $P(x)\in\mathbb{Q}[x]$,
which concludes the proof of the lemma.
\end{proof}

\begin{theorem}[{Lindemann \cite{Li}}]\label{T31} 
$\pi$ is a transcendental number.
\end{theorem}

\begin{proof}
In the following we use a well known idea (see for instance 
\cite[Theorem 1.3]{B}). Let us assume that $\pi$ is an algebraic number, that is, $\pi
\in\overline{\mathbb{Q}}$. Since $i=\sqrt{-1}\in\overline{\mathbb{Q}}$ and
$\overline{\mathbb{Q}}$ is a field (Lemma \ref{L21}), we see that $i\pi$ is
also in $\overline{\mathbb{Q}}$, that is, $\alpha=\alpha_{1}=i\pi$ is a root
of the minimal polynomial $f_{\alpha}\in\mathbb{Q}[x]$ with its roots
$\alpha_{1},\alpha_{2},\ldots,\alpha_{n}$ $\in\overline{\mathbb{Q}}$. Since
$e^{\alpha_{1}}+1=0$ (Euler's formula), we see that the following number,%
\[
U=(e^{\alpha_{1}}+1)\cdot(e^{\alpha_{2}}+1)\cdots(e^{\alpha_{n}}+1)
\]
is equal to zero, that is,%
\begin{equation}
U=1+\sum_{i=1}^{n}e^{\alpha_{i}}+\sum_{i,j=1,i<j}^{n}e^{\alpha_{i}+\alpha_{j}%
}+\sum_{i,j,k=1,i<j<k}^{n}e^{\alpha_{i}+\alpha_{j}+\alpha_{k}}+\cdots
+e^{\sum_{i=1}^{n}\alpha_{i}}=0. \label{3.11}%
\end{equation}
We denote $\mathcal{M}_{1}=\{\alpha_{1},\alpha_{2},\ldots,\alpha_{n}\}$,
$\mathcal{M}_{2}=\{\alpha_{i}+\alpha_{j}:1\leq i<j\leq n\}$, $\mathcal{M}%
_{3}=\{\alpha_{i}+\alpha_{j}+\alpha_{k}:1\leq i<j<k\leq n\}$,\ldots,
$\mathcal{M}_{n}=\{\sum_{i=1}^{n}\alpha_{i}\}$. We see that for any embedding
$\sigma:\mathbb{Q}[\alpha_{1},\alpha_{2},\ldots,\alpha_{n}]\rightarrow
\mathbb{C}$, we have $\sigma(\mathcal{M}_{s})=$ $\mathcal{M}_{s}$ for any
$s=1,2,\dots,n$. Thus, for any $s=1,2,\ldots,n$, there exists a polynomial
$P_{s}\in\mathbb{Q}[x]$ with the set of its roots exactly $\mathcal{M}_{s}$.
Therefore, the set of all roots of the polynomial $P=P_{1}\cdot P_{2}%
\cdots P_{n}$ is $\mathcal{M}=\cup_{s=1}^{n}\mathcal{M}_{s}$, that
is, all the powers of $e$ in formula (\ref{3.11}). Let $d$ be the degree of
$P$ and let $\mathcal{M}_{0}=\{\beta_{1},\beta_{2},\ldots,\beta_{t}\}$ be the
set of all nonzero (distinct or not) roots of $P.$ Thus,%
\begin{equation}
U=k+e^{\beta_{1}}+e^{\beta_{2}}+\cdots+e^{\beta_{t}}=0, \label{3.12}%
\end{equation}
where $k=d+1-t\geq1$. It is easy to see that $\sigma(\mathcal{M}%
_{0})=\mathcal{M}_{0}$ for any embedding $\sigma$ of $\mathbb{Q}[\alpha
_{1},\alpha_{2},\ldots,\alpha_{n}]$ in $\mathbb{C}$.

Now, we take a positive integer $c$ such that $c\beta_{1},c\beta_{2}%
,\ldots,c\beta_{t}$ are algebraic integers (Remark ~\ref{R51}) and let us
define the following \textit{Hermite-Lindemann type polynomial},%
\[
H(x)=c^{t}\prod_{j=1}^{t}(x-\beta_{j})\in\mathbb{A}[x],
\]
that is, its coefficients are algebraic integers.

Now, for any prime number $p$, we define a \textit{Hermite type polynomial }of
degree $d_{p}=p(t+1)-1$,
\[
f_{p}(x)=\frac{c^{p-1}}{(p-1)!}x^{p-1}[H(x)]^{p}\in\frac{1}{(p-1)!}%
\mathbb{A}[x].
\]
It is not difficult to calculate all the derivatives of $f_{p}$ (up to the
order $d_{p}$) at the points $\beta_{1},\beta_{2},\ldots,\beta_{t}$ and $0$.
Thus, we find that
\begin{equation}
f_{p}^{(h)}(\beta_{i})=%
\begin{cases}
0, & \text{$0\leq h<p$,}\\[4pt]%
pN_{i,h,p}, & \text{$h\geq p$}%
\end{cases}
\label{3.13}\\
\quad\text{for }i=1,2,\ldots,t,
\end{equation}
where $N_{i,h,p}\in\mathbb{A}$, and
\begin{equation}
f_{p}^{(h)}(0)=%
\begin{cases}
0, & 0\leq h<p-1,\\[4pt]%
c^{p-1}[H(0)]^{p}, & h=p-1,\\[4pt]%
pM_{p,h}, & h\geq p,
\end{cases}
\label{3.14}%
\end{equation}
where $c^{p-1}[H(0)]^{p},M_{p,h}\in\mathbb{A}$.

Following Hermite's basic idea \cite{H} we consider the following complex
integrals of analytic functions over the segment $[0,\beta_{i}]$,
$i=1,2,\ldots,t$,%
\begin{equation}
I_{p}(i)=\int_{0}^{\beta_{i}}f_{p}(x)e^{-x}\,\mathrm{d}x. \label{3.15}%
\end{equation}
We integrate by parts $d_{p}$ times in formula (\ref{3.15}) and we find the
following new formula,%
\begin{equation}
I_{p}(i)=\left.  -e^{-x}\left[  f_{p}(x)+f_{p}^{\prime}(x)+f_{p}^{\prime
\prime}(x)+\cdots+f_{p}^{(d_{p})}(x)\right]  \right\vert _{0}^{\beta_{i}}.
\label{3.16}%
\end{equation}
If we denote%
\[
F_{p}(x)=f_{p}(x)+f_{p}^{\prime}(x)+f_{p}^{\prime\prime}(x)+\cdots
+f_{p}^{(d_{p})}(x),
\]
formula (\ref{3.16}) becomes,%
\begin{equation}
I_{p}(i)=-e^{-\beta_{i}}F_{p}(\beta_{i})+F_{p}(0). \label{3.17}%
\end{equation}
From (\ref{3.13}) and (\ref{3.14}) we find,%
\begin{equation}
\sum_{i=1}^{r}e^{\beta_{i}}I_{p}(i)=pN-pM-kc^{p-1}[H(0)]^{p}, \label{3.18}%
\end{equation}
where $N$, $M$ and $kc^{p-1}[H(0)]^{p}$ are algebraic integers. But
$kc^{p-1}[H(0)]^{p}$ is not divisible by $p$ for $p$ large enough (Lemma
~\ref{L71}), so that the right side of equality (\ref{3.18}) is not zero for
$p$ a sufficiently large prime number. Since the left side of equation
(\ref{3.18}) is a symmetric expression relative to $\beta_{1},\beta_{2}%
,\ldots,\beta_{t}$, we obtain that
\[
\sigma\left(  pN-pM-kc^{p-1}[H(0)]^{p}\right)  =pN-pM-kc^{p-1}[H(0)]^{p}%
\]
for any embedding $\sigma$ of $\mathbb{Q}[\beta_{1},\beta_{2},\ldots,\beta
_{r}]$ in $\mathbb{C}$. Hence, $pN-pM-kc^{p-1}[H(0)]^{p}\in\mathbb{Q}$
(Corollary ~\ref{C91}). But $pN-pM-kc^{p-1}[H(0)]^{p}$ is also an algebraic
integer. Therefore, it is an integer (Remark ~\ref{R41}), that is, the right
side is a set of nonzero integers for $p$ large enough (say $p>p_{0}$).
Consequently, the set $\left\{  \sum_{i=1}^{t}e^{\beta_{i}}I_{p}(i)\right\}
_{p>p_{0}}$ cannot have $0$ as a limit point. But $I_{p}(i)\rightarrow0$, if
$p\rightarrow\infty$ as a prime number for $i=1,2,\ldots,t$. Indeed, let
$T_{i},$ $V_{i},$ and $W_{i}$ be the greatest value of $\left\vert
x\right\vert $, $\left\vert H(x)\right\vert $ and $\left\vert e^{-x}%
\right\vert $ respectively on the segment $[0,\beta_{i}]$, $i=1,2,\ldots,t$.
Thus,%
\[
\left\vert I_{p}(i)\right\vert \leq\left\vert \beta_{i}\right\vert W_{i}%
V_{i}\frac{[cT_{i}V_{i}]^{p-1}}{(p-1)!}\rightarrow0,
\]
if $p\rightarrow\infty$ on the set of prime numbers. Thus, the set $\left\{
\sum_{i=1}^{r}e^{\beta_{i}}I_{p}(i)\right\}  _{p>p_{0}}$ has $0$ as a limit
point, a contradiction. In conclusion, $\pi$ cannot be an algebraic number.
\end{proof}

\begin{corollary}
\label{CA1} The mapping $z\rightarrow e^{z}$ is a one-to-one mapping on the
field $\overline{\mathbb{Q}}$ of algebraic numbers.
\end{corollary}

\begin{proof}
Let $\alpha,\beta$ be two algebraic numbers such that $e^{\alpha}=e^{\beta}$.
Thus, $\alpha-\beta=2k\pi i$, where $i=\sqrt{-1}$. If $k$ is not zero, then
$\pi=(\alpha-\beta)/2ki\in\overline{\mathbb{Q}}$, a contradiction (Theorem
~\ref{T31}). Therefore, $k=0$, that is, $\alpha=\beta$.
\end{proof}

\section{Lindemann-Weierstrass Theorem}\label{Section3}

We continue to use the same definitions and notation from Section 1 and we
start with a special case of the main result, namely with a fundamental result
of Lindemann \cite{Li1}. The proof is completely different from the original
one, but it includes some of the great ideas of Hermite \cite{H}, Lindemann
\cite{Li1} and Baker \cite{B}. However, during the proof we assume that $\pi$
is a transcendental number (Theorem ~\ref{T31}).

\begin{theorem}[{Lindemann \cite{Li1}}]\label{T21} 
Let $\alpha_{1},\alpha_{2}%
,\ldots,\alpha_{n}\in\overline{\mathbb{Q}}$ be $n$ $(n\geq1)$ distinct
algebraic numbers. Then the complex numbers $e^{\alpha_{1}},e^{\alpha_{2}%
},\ldots,e^{\alpha_{n}}$ are linear independent over~$\mathbb{Q}$, that is,
if
\begin{equation}
b_{1}e^{\mathbb{\alpha}_{1}}+b_{2}e^{\mathbb{\alpha}_{2}}+\cdots
+b_{n}e^{\mathbb{\alpha}_{n}}=0, \label{201}%
\end{equation}
with $b_{1},b_{2},\ldots,b_{n}\in\mathbb{Q}$, then $b_{1}=b_{2}=\cdots
=b_{n}=0$.
\end{theorem}

\begin{proof}
Let us assume the opposite, namely that there exist $b_{1},b_{2},\ldots
,b_{n}\in\mathbb{Q}$, not all zero, so that equality (\ref{201}) is true. We
can also assume that all $b_{1},b_{2},\ldots,b_{n}$ are nonzero integers.

Let $L$ be the subfield of $\mathbb{C}$ generated by $\alpha_{1},\alpha
_{2},\ldots,\alpha_{n}$ and their conjugates (over~$\mathbb{Q}$), and let
$m=L:\mathbb{Q}$ be the dimension of $L$ over $\mathbb{Q}$ as a vector space.
Remark ~\ref{R701} and Lemma ~\ref{L41} say that $L$ is a normal extension of
$\mathbb{Q}$ and $G=Gal(L/\mathbb{Q})$ has $m$ elements, $\sigma
_{1}=id.,\sigma_{2},\ldots,\sigma_{m}$. First of all, we are able to note that
$e^{\sigma_{s}(\alpha_{1})},e^{\sigma_{s}(\alpha_{2})},\ldots,e^{\sigma
_{s}(\alpha_{n})}$ are distinct complex numbers for any $s=1,2,\ldots,m$.
Indeed, if $e^{\sigma_{s}(\alpha_{i})}=e^{\sigma_{s}(\alpha_{j})}$ for $i\neq
j$, $i$, $j=1,2,\ldots,n$, then $\sigma_{s}(\alpha_{i})=\sigma_{s}(\alpha
_{j})$ (Corollary ~\ref{CA1}).\ Since $\sigma_{s}$ is an automorphism of $L$,
we see that $\alpha_{i}=\alpha_{j}$, a contradiction ($\alpha_{1},\alpha
_{2},\ldots,\alpha_{n}$ are distinct).

Let us write
\begin{equation}\label{202}
\begin{cases}
    Q_{1} = b_{1}e^{\mathbb{\alpha}_{1}}+b_{2}e^{\mathbb{\alpha}_{2}}+\cdots
+b_{n}e^{\mathbb{\alpha}_{n}}=0,\\
   Q_{2} = b_{1}e^{\sigma_{2}(\mathbb{\alpha}_{1})}+b_{2}e^{\sigma
_{2}(\mathbb{\alpha}_{2})}+\cdots+b_{n}e^{\sigma_{2}(\mathbb{\alpha}_{n})},\\
   \phantom{Q_2\;\; }\vdots\\
   Q_{m} = b_{1}e^{\sigma_{m}(\mathbb{\alpha}_{1})}+b_{2}e^{\sigma_{m}(\mathbb{\alpha}_{2})}+\cdots+b_{n}e^{\sigma_{m}(\mathbb{\alpha}_{n})},
\end{cases}
\end{equation}
and we see that
\[
R=Q_{1}\cdot Q_{2}\cdots Q_{m}=0.
\]
Now, we define
\[
S=\{\mathbf{k}=(k_{1},k_{2},\ldots,k_{m})\in\mathbb{N}^{\ast}:1\leq k_{j}\leq
n,1\leq j\leq m\}.
\]
Thus,
\begin{equation}
R=\sum_{\mathbf{k}\in S}b_{k_{1}}b_{k_{2}}\cdots b_{k_{m}}e^{\sum_{i=1}^{m}%
\sigma_{i}(\alpha_{k_{i}})}=0.\label{204}%
\end{equation}
Let us write $\overline{k}=\sum_{i=1}^{m}\sigma_{i}(\alpha_{k_{i}})$,
$b_{\mathbf{k}}=b_{k_{1}}b_{k_{2}}\cdots b_{k_{m}}$ and let us consider
another element $\mathbf{l}=(l_{1},l_{2},\ldots,l_{m})\in S$ such that
$\overline{k}=\overline{l}$. Then, we give common factor $e^{\overline{k}}$ in
formula (\ref{204}) and substitute the coefficient $b_{\mathbf{k}}$ in front
of $e^{\overline{k}}$ with $b_{\mathbf{k}}+b_{\mathbf{l}}$, and so on. We
continue to do this until all the powers $\overline{k}$ of $e^{\overline{k}}$
are distinct. Let us denote these last distinct powers by $\beta_{1},\beta
_{2},\ldots,\beta_{t}$ and by $b_{1}^{\ast},b_{2}^{\ast},\ldots,b_{t}^{\ast}$
the new coefficients of $e^{\beta_{1}},e^{\beta_{2}},\ldots,e^{\beta_{t}}$
respectively. We denote $W=\{\beta_{1},\beta_{2},\ldots,\beta_{t}\}$. Thus,
formula (\ref{204}) becomes,%
\begin{equation}
R=\sum_{j=1}^{t}b_{j}^{\ast}e^{\beta_{j}}=0.\label{205}%
\end{equation}
In formula (\ref{205}) not all $b_{j}^{\ast}$ are zero. Indeed, let us return
to formula (\ref{202}) and choose $\mathbf{l}=(l_{1},l_{2},\ldots,l_{m})$ so
that for any $s=1,2,\dots,m,$ $\sigma_{s}(\alpha_{l_{s}})$ is the greatest
element of the set $\{\sigma_{s}(\alpha_{1}),\sigma_{s}(\alpha_{2}%
),\ldots,\sigma_{s}(\alpha_{n})\}$ relative to the lexicographic order in
$(\mathbb{C}$,$+)$. Thus, $0\neq b_{l_{1}}b_{l_{2}}\cdots b_{l_{m}}$ cannot
cancel out with other $b_{h_{1}}b_{h_{2}}\cdots b_{h_{m}}$, except the case
when
\[
e^{\sum_{i=1}^{m}\sigma_{i}(\alpha_{l_{i}})}=e^{\sum_{i=1}^{m}\sigma
_{i}(\alpha_{h_{i}})}.
\]
But this situation cannot appear because $z\rightarrow e^{z}$ is one-to-one on
$\overline{\mathbb{Q}}$ (Corollary ~\ref{CA1}).

We can also assume now that $b_{1}^{\ast},b_{2}^{\ast},\ldots,b_{t}^{\ast}$
are all nonzero elements in $\mathbb{Q}$. Moreover, because in formula
(\ref{205}) $b_{j_{0}}^{\ast}\neq0$ implies that the coefficient $b_{k_{0}%
}^{\ast}$, which is in front of $e^{\beta_{k_{0}}}$, is also a nonzero element
for $\beta_{k_{0}}=\sigma(\beta_{j_{0}})$, $\sigma\in G$, we see that
$\sigma(W)=W$ for any $\sigma\in G$.

Let $c$ be a positive integer so that $c\beta_{1},c\beta_{2},\dots,c\beta_{t}$
are algebraic integers (Remark ~\ref{R51}), and let us consider the
\textit{Hermite-Lindemann polynomial},%
\[
H(x)=c^{t}\prod_{i=1}^{t}(x-\beta_{i}),
\]
where $\beta_{1},\beta_{2},\ldots,\beta_{t}$ are the distinct elements of $W$
defined above. It is clear that the coefficients of the polynomial $H$ are
algebraic integers, that is, $H\in\mathbb{A}[x]$ (we recall that
$\mathbb{A\subset}\overline{\mathbb{Q}}$ is the subring of all algebraic
integers in $\overline{\mathbb{Q}}$, the field of all algebraic numbers in
$\mathbb{C}$ (Lemma ~\ref{L21} and Lemma ~\ref{L61}).

For any fixed prime number $p$ and $j\in\{1,2,\ldots,t\}$ we define the
\textit{Hermite polynomial}:%
\begin{equation}
f_{p,j}(x)=\frac{1}{(p-1)!}\frac{[H(x)]^{p}}{x-\beta_{j}}\in\frac{1}%
{(p-1)!}\mathbb{A}[x]. \label{208}%
\end{equation}
We denote by $d_{p}=tp-1$ its degree and we easily see that it is independent
on $j\in\{1,2,\ldots,t\}$. We also write%
\begin{equation}
U_{j}(x)=c^{t-1}\prod_{i=1,i\neq j}^{t}(x-\beta_{i})\in\mathbb{A}[x],
\label{209}%
\end{equation}
with $j=1,2,\ldots,t$. Thus, the formula (\ref{208}) becomes%
\[
f_{p,j}(x)=\frac{c^{p-1}(x-\beta_{j})^{p-1}[U_{j}(x)]^{p}}{(p-1)!}.
\]
By a careful use of Leibniz's differentiation rule for products of functions,
we obtain%
\begin{equation}
f_{p,j}^{(k)}(\beta_{i})=M_{k,j,i}^{(p)}\cdot p,\text{ }i=1,2,\ldots,t,i\neq
j,k=0,1,\ldots,d_{p} \label{211}%
\end{equation}
and $M_{k,j,i}^{(p)}\in\mathbb{A}$, that is, it is an algebraic integer. For
$i=j$ and $k\neq p-1$, we obtain%
\begin{equation}
f_{p,j}^{(k)}(\beta_{j})=N_{k,j}^{(p)}\cdot p,\text{ }k=0,1,\dots,d_{p},
\label{212}%
\end{equation}
where $N_{k,j}^{(p)}\in\mathbb{A}$. Finally, for $i=j$ and $k=p-1$, we also
have%
\begin{equation}
f_{p,j}^{(p-1)}(\beta_{j})=c^{p-1}\left[  U_{j}(\beta_{j})\right]  ^{p},
\label{213}%
\end{equation}
which is a nonzero algebraic integer.

Let us consider now a Hermite type complex integral:%
\begin{equation}
\label{214}I_{p,j}(i)=\int_{\beta_{j}}^{\beta_{i}}f_{p,j}(x)e^{-x}\,
\mathrm{d}x,
\end{equation}
where $p$ is a prime number, and $i,j\in\{1,2,\ldots,t\}$. Since the integrand
is an analytic function on $\mathbb{C}$, the complex integral can be
calculated on any path of class~$C^{1}$ which connects the points $\beta_{j}$,
$\beta_{i}$ in the complex plane $\mathbb{C}$. For instance, we can take the
segment $[\beta_{j},\beta_{i}]$ as a path which connect $\beta_{j}$ and
$\beta_{i}$.

Now we write%
\begin{equation}
F_{p,j}(x)=f_{p,j}(x)+f_{p,j}^{\prime}(x)+f_{p,j}^{\prime\prime}%
(x)+\cdots+f_{p,j}^{(d_{p})}(x), \label{215}%
\end{equation}
and we integrate formula (\ref{214}) $d_{p}$ times by parts. We see that%
\begin{equation}
I_{p,j}(i)=-e^{-\beta_{i}}F_{p,j}(\beta_{i})+e^{-\beta_{j}}F_{p,j}(\beta_{j}).
\label{216}%
\end{equation}

Now, we make the following notations:%
\[
T_{ij}=\underset{x\in\lbrack\beta_{j},\beta_{i}]}{\sup}\left\vert
U_{j}(x)\right\vert ,\text{ }T^{ij}=\underset{x\in\lbrack\beta_{j},\beta_{i}%
]}{\sup}\left\vert e^{-x}\right\vert .
\]
Thus,%
\[
\left\vert I_{p,j}(i)\right\vert \leq c\left\vert \beta_{i}-\beta
_{j}\right\vert \cdot T_{ij}\cdot T^{ij}\cdot\frac{\left[  c\left\vert
\beta_{i}-\beta_{j}\right\vert \cdot T_{ij}\right]  ^{p-1}}{(p-1)!}.
\]
For any $i,j\in\{1,2,\ldots,t\}$, we see that $I_{p,j}(i)\rightarrow0$, if
$p\rightarrow\infty$ as a prime number.

Now, we return to formula (\ref{216}) and we evaluate its right side. From
formulas (\ref{215}), (\ref{211}), (\ref{212}) and (\ref{213}), we see that
$F_{p,j}(\beta_{i})=M_{j,i}^{[p]}\cdot p,$ where $i\neq j,$ $M_{j,i}^{[p]}%
\in\mathbb{A}$, and $F_{p,j}(\beta_{j})=N_{j,i}^{[p]}\cdot p+c^{p-1}\left[
U_{j}(\beta_{j})\right]  ^{p}\in\mathbb{A}$ ($N_{j,i}^{[p]}\in\mathbb{A}$),
which is not zero $\operatorname{mod}p$ for~$p$ large enough (each
$U_{j}(\beta_{j})$, $j\in\{1,2,\ldots,t\}$, is not zero (formula (\ref{209}))
and it has only a finite number of prime divisors (Lemma ~\ref{L71}))$.$
Hence, for any $j=1,2,\ldots,t$, $F_{p,j}(\beta_{j})$ is not zero if $p$ is
large enough.

Now, we fix a $j\in\{1,2,\ldots,t\}$, we go back to formula (\ref{216}) and we
calculate,%
\[
S_{j}(p)=\sum_{i=1}^{t}b_{i}^{\ast}e^{\beta_{i}}I_{p,j}(i)=-\sum_{i=1}%
^{t}b_{i}^{\ast}F_{p,j}(\beta_{i})+e^{-\beta_{j}}F_{p,j}(\beta_{j})\sum
_{i=1}^{t}b_{i}^{\ast}e^{\beta_{i}}.
\]
Since $R=\sum_{i=1}^{t}b_{i}^{\ast}e^{\beta_{i}}=0$ (formula (\ref{205})), we
obtain%
\begin{equation}
S_{j}(p)=-\sum_{i=1}^{t}b_{i}^{\ast}F_{p,j}(\beta_{i}). \label{2.117}%
\end{equation}
Since $S_{j}(p)=M_{j}(p)\cdot p+b_{j}^{\ast}c^{p-1}\left[  U_{j}(\beta
_{j})\right]  ^{p}\in\mathbb{A}$ ($M_{j}(p)\in\mathbb{A}$) is not zero for any
$j=1,2,\ldots,t$ and for sufficiently large prime numbers $p$, we finally can
see that $S(p)=\prod_{j=1}^{t}S_{j}(p)$ is not zero for $p$ large enough.

We can see that $S(p)$ is a symmetric polynomial relative to $\beta_{1}%
,\beta_{2},\ldots,\beta_{t}$. Thus, since $S(p)\in L$ and because
$\sigma(S(p))=S(p)$ ($\sigma$ acts as a permutation on $W$) for all $\sigma\in
G=Gal(L/\mathbb{Q})$, Corollary ~\ref{C91} says that $S(p)\in\mathbb{Q}$, that
is, $S(p)\in\mathbb{Z}$, $S(p)$ being an algebraic integer (see formulas
(\ref{211}), (\ref{212}), (\ref{215}), and (\ref{2.117})). Therefore,
$S(p)\in\mathbb{Z}\setminus\{0\}$ for a sufficiently large prime number $p$.
At the same time, as we saw above, $S(p)\rightarrow0$, if $p\rightarrow\infty$
through the set of prime numbers, a contradiction. In conclusion, the
statement of the theorem is true.
\end{proof}

\begin{corollary}
\label{C23} For any nonzero real algebraic number $\alpha,$ $e^{\alpha}$ is an
irrational number.
\end{corollary}

\begin{proof}
Let us suppose that $e^{\alpha}$ is a rational number $q$. Let us take in
Theorem ~\ref{T21} $n=2,$ $\alpha_{1}=0$, $\alpha_{2}=\alpha$, $b_{1}=-q$ and
$b_{2}=1$. Since $\alpha_{1}\neq\alpha_{2}$ and $b_{1}e^{\alpha_{1}}%
+b_{2}e^{\alpha_{2}}=0$, we just obtained a contradiction relative to the
statement of Theorem~\ref{T21}. Therefore, $e^{\alpha}$ is an irrational
number, and the proof is finished.
\end{proof}

In 1885 K. Weierstrass \cite{W} managed to generalize and make some
improvements to Lindemann's Theorem ~\ref{T21} as follows. In our proof we use
again that $\pi$ is a transcendental number (Theorem ~\ref{T31}).

\begin{theorem}[Lindemann-Weierstrass]\label{T22} 
Let $\alpha_{1},\alpha_{2}%
,\ldots,\alpha_{n}\in\overline{\mathbb{Q}}$ be $n$ $(n\geq1)$ distinct
algebraic numbers. Then the complex numbers $e^{\alpha_{1}},e^{\alpha_{2}%
},\ldots,e^{\alpha_{n}}$ are linear independent over $\overline{\mathbb{Q}}$,
the field of all algebraic numbers.
\end{theorem}

\begin{proof}
The idea is to reduce the proof of this theorem to the proof of Theorem
\ref{T21}. We assume the opposite, namely that there exist $c_{1},c_{2}%
,\ldots,c_{n}\in\overline{\mathbb{Q}}$, not all zero, such that
\begin{equation}
c_{1}e^{\alpha_{1}}+c_{2}e^{\alpha_{2}}+\cdots+c_{n}e^{\alpha_{n}}=0.
\label{218}%
\end{equation}
Moreover, we can assume that all $c_{i},$ $i=1,2,\ldots,n$, are not zero.
Multiplying the equality (\ref{218}) by an appropriate positive integer, we
can also suppose that $c_{1},c_{2},\ldots,c_{n}$ are algebraic integers, that
is $c_{1},c_{2},\ldots,c_{n}\in\mathbb{A}$ (Remark ~\ref{R51}). Let
$\mathcal{O}(c_{j})$ (the orbit of $c_{j}$) be the set of all conjugates
(relative to $\mathbb{Q}$) of $c_{j}$, $j=1,2,\ldots,n$, and let
$\mathcal{N}=\bigcup_{j=1}^{n}\mathcal{O}(c_{j})$ be the union (not necessarily
disjoint, that is, it is possible that some of these orbits coincide) of the
orbits $\mathcal{O}(c_{j})$, $j=1,2,\ldots,n$. We denote $L=$ $\mathbb{Q}%
[\mathcal{N}]$, the least subfield of $\overline{\mathbb{Q}}$ generated by all
the elements of $\mathcal{N}$. It is not difficult to see that any
$\mathbb{Q}$-embedding $\sigma$ of $L$ into $\mathbb{C}$ has values in $L$
itself. Thus, all these $\mathbb{Q}$-embeddings, $\sigma_{1}=id.,\sigma
_{2},\ldots,\sigma_{m}$, are exactly the elements of the Galois group
$J=Gal(L/\mathbb{Q})$. Here $m=L:\mathbb{Q}$ (Lemma ~\ref{L41}).

Starting with formula (\ref{218}), we can define the following complex
numbers:%
\begin{equation}\label{219}
\begin{cases}
  V_{1} = \sigma_{1}(c_{1})e^{\alpha_{1}}+\sigma_{1}(c_{2})e^{\alpha_{2}}
    +\cdots+\sigma_{1}(c_{n})e^{\alpha_{n}}=0,\\
  V_{2} = \sigma_{2}(c_{1})e^{\alpha_{1}}+\sigma_{2}(c_{2})e^{\alpha_{2}}
       +\cdots+\sigma_{2}(c_{n})e^{\alpha_{n}},\\
       \phantom{V_2\;\; }\vdots\\
  V_{m} = \sigma_{m}(c_{1})e^{\alpha_{1}}+\sigma_{m}(c_{2})e^{\alpha_{2}}
     +\cdots+\sigma_{m}(c_{n})e^{\alpha_{n}}.
\end{cases}
\end{equation}
%
Since each $c_{1},c_{2},\ldots,c_{n}$ is not zero, we see that each
$\sigma_{j}(c_{1}),\sigma_{j}(c_{2}),\ldots,\sigma_{j}(c_{n})$, $j=~1,2,\ldots
,m$, is not zero.

Now, we see that%
\[
V=V_{1}\cdot V_{2}\cdots V_{m}=0.
\]
If we define,%
\[
S^{\ast}=\{\mathbf{\alpha}=(\alpha_{i_{1}},\alpha_{i_{2}},\ldots,\alpha
_{i_{m}}):1\leq i_{j}\leq n,\text{ }j=1,2,\ldots,m\},
\]
we finally obtain,%
\begin{equation}
0=V=\sum_{\mathbf{\alpha}\in S^{\ast}}\sigma_{1}(c_{i_{1}})\cdot\sigma
_{2}(c_{i_{2}})\cdots\sigma_{m}(c_{i_{m}})\cdot e^{\sum_{j=1}%
^{m}\alpha_{i_{j}}}. \label{221}%
\end{equation}
Now, among all the sums $\gamma=\sum_{j=1}^{m}\alpha_{i_{j}}$, let us choose
the distinct ones: $\gamma_{1},\gamma_{2},\ldots,\gamma_{t}$. Thus, in formula
(\ref{221}) we can write $e^{\gamma_{j}},$ $j=1,2,\ldots,t$, only once, and
denote by~$c_{j}^{\ast}$ the coefficient which appears in front of this
$e^{\gamma_{j}}$. Thus, this coefficient is%
\begin{equation}
c_{j}^{\ast}=\sum_{\mathbf{\alpha}\in S^{\ast},\text{ }\sum_{s=1}^{m}%
\alpha_{i_{s}}=\gamma_{j}}\sigma_{1}(c_{i_{1}})\cdot\sigma_{2}(c_{i_{2}}%
)\cdots\sigma_{m}(c_{i_{m}}). \label{222}%
\end{equation}
Thus, formula (\ref{221}) can also be written as%
\begin{equation}
0=V=\sum_{j=1}^{t}c_{j}^{\ast}e^{\gamma_{j}}, \label{223}%
\end{equation}
where $\gamma_{1},\gamma_{2},\ldots,\gamma_{t}$ are distinct algebraic
numbers, and $c_{j}^{\ast}$, $j=~1,2,\ldots,t$, are algebraic integers
calculated as in formula (\ref{222}). We state that not all $c_{j}^{\ast}$,
$j=1,2,\ldots,t$, are zero. Indeed, take in $(\mathbb{C}$, $+)$ the
lexicographic order and denote by~$\alpha_{j_{0}}$ the greatest element in the
set $\{\alpha_{1},\alpha_{2},\ldots,\alpha_{n}\}$. This $\alpha_{j_{0}}$ is
unique, because $\alpha_{1},\alpha_{2},\ldots,\alpha_{n}$ are distinct. Thus
$\gamma_{j_{0}}=\sum_{s=1}^{m}\alpha_{j_{0}}=m\alpha_{j_{0}}$ and%
\[
c_{j_{0}}^{\ast}=\sigma_{1}(c_{j_{0}})\cdot\sigma_{2}(c_{j_{0}})\cdot
\cdots\cdot\sigma_{m}(c_{j_{0}})\neq0,
\]
where $c_{j_{0}}^{\ast}$ is the coefficient of $e^{\gamma_{j_{0}}}$ in formula
(\ref{223}). There is no other product
\[
\sigma_{1}(c_{i_{1}})\cdot\sigma_{2}(c_{i_{2}})\cdots\sigma
_{m}(c_{i_{m}})
\]
to cancel out with $c_{j_{0}}^{\ast}=\sigma_{1}(c_{j_{0}})\cdot\sigma
_{2}(c_{j_{0}})\cdots\sigma_{m}(c_{j_{0}})$, because $\gamma_{j_{0}%
}$ is unique and the mapping $z\rightarrow e^{z}$ is one-to-one (Corollary
\ref{CA1}).

From formula (\ref{222}), we see that $\sigma(c_{j}^{\ast})=c_{j}^{\ast}$ for
any $\sigma\in J=Gal(L/\mathbb{Q})$ and for any $j=1,2,\ldots,t$. Indeed, let
us fix a $\gamma_{j}=\sum_{s=1}^{m}\alpha_{i_{s}}$. The coefficient
$c_{j}^{\ast}$ of $e^{\gamma_{j}}$ is a sum of the type given in formula
(\ref{222}). If this sum contains the term $\sigma_{1}(c_{i_{1}})\cdot
\sigma_{2}(c_{i_{2}})\cdots\sigma_{m}(c_{i_{m}})$, it also contains
the term
\begin{equation}
\sigma(\sigma_{1}(c_{i_{1}}))\cdot\sigma(\sigma_{2}(c_{i_{2}}))\cdots
\sigma(\sigma_{m}(c_{i_{m}})), \label{552}%
\end{equation}
where $\sigma\in J.$ Indeed, because $\{\sigma\sigma_{1},$ $\sigma\sigma
_{2},\ldots,\sigma\sigma_{m}\}$ is simply a permutation of $\{\sigma_{1},$
$\sigma_{2},\ldots,\sigma_{m}\}$, $\alpha_{i_{1}}$ appears as a power of $e$
in formula (\ref{219}) on the row $\sigma\sigma_{1}$, $\alpha_{i_{2}}$ appears
as a power of $e$ on the row $\sigma\sigma_{2}$, and so on. This means that
$\gamma_{j}=\sum_{s=1}^{m}\alpha_{i_{s}}$ does not change, that is, the
product (\ref{552}) is a term of the same sum $c_{j}^{\ast}$. Hence
$\sigma(c_{j}^{\ast})=c_{j}^{\ast}$ for any $\sigma\in J$, and for any
$j=1,2,\ldots,t$. Hence, from Corollary ~\ref{C91}, we can conclude that
$c_{j}^{\ast}\in\mathbb{Q}$ for any $j=1,2,\ldots,t$. Coming back to formula
(\ref{223}), we see that $\sum_{j=1}^{t}c_{j}^{\ast}e^{\gamma_{j}}=0$ is a
nontrivial (not all $c_{j}^{\ast}$ are zero) null linear combination with
coefficients in $\mathbb{Q}$. But this contradicts the statement of Theorem
~\ref{T21}. Thus the assertion of Theorem ~\ref{T22} is fully proved.
\end{proof}

\begin{corollary}[Lindemann~\cite{Li}, or {\cite[Theorem 7.2]{P}}]\label{C22} 
For any nonzero
algebraic number $\alpha$, $e^{\alpha}$ is a transcendental number.
\end{corollary}

\begin{proof}
Assume that $e^{\alpha}=\beta$ is an algebraic number. Since $\alpha\neq0$ and
since we see that $1\cdot e^{\alpha}-\beta\cdot e^{0}=0$, from Theorem
~\ref{T22}, we get $1=\beta=0$, a contradiction. Therefore, $e^{\alpha}$ is a
transcendental number.
\end{proof}

\subsection*{Acknowledgements}
We are grateful to the members of the \emph{\textquotedblleft Nicolae
Popescu\textquotedblright\ Algebra and Number Theory Seminar} at the
\emph{\textquotedblleft Simion Stoilow\textquotedblright\ Institute of
Mathematics of the Romanian Academy} for fruitful discussion related to the
subject of this paper. We are also grateful to the unknown referees for their
useful remarks on the previous version of this paper.

\end{document}